\documentclass[preprint]{elsarticle}

\usepackage{lineno,hyperref}
\modulolinenumbers[5]

\usepackage{amsmath,amsthm,amssymb}
\usepackage{kbordermatrix}
\usepackage{enumitem}
\usepackage{physics}
\usepackage{mathtools}
\usepackage{mathrsfs}

\allowdisplaybreaks

\newtheorem{thm}{Theorem}[section]
\newtheorem{lem}[thm]{Lemma}
\newtheorem{cor}[thm]{Corollary}
\newtheorem{prop}[thm]{Proposition}
\newtheorem{conj}[thm]{Conjecture}
\newtheorem{obs}[thm]{Observation}

\theoremstyle{definition}			                						% Subseq. Thm Style
\newtheorem{mydef}[thm]{Definition}
\newtheorem{rem}[thm]{Remark}

\newtheorem{prob}[thm]{Problem}

\DeclareMathOperator{\diag}{diag}
\DeclareMathOperator{\circulant}{circ}

\journal{Linear Algebra and its Applications}

\begin{document}

\begin{frontmatter}

\title{Polynomials that preserve nonnegative matrices}
% \title{Polynomials that preserve nonnegative matrices\tnoteref{mytitlenote}}
% \tnotetext[mytitlenote]{Fully documented templates are available in the elsarticle package on \href{http://www.ctan.org/tex-archive/macros/latex/contrib/elsarticle}{CTAN}.}

% %% Group authors per affiliation:
% \address{Radarweg 29, Amsterdam}
% \fntext[myfootnote]{Since 1880.}

\author[uwb]{Benjamin J.~Clark\fnref{funding}}
\ead{thrin@uw.edu}
\fntext[funding]{Supported by a 2018 University of Washington Bothell Founder's Fellow Award.}

\author[uwbaddress]{Pietro Paparella\fnref{funding2}\corref{ca}}
\cortext[ca]{Corresponding author}
\ead{pietrop@uw.edu}
\ead[url]{http://faculty.washington.edu/pietrop/}

\address[uwbaddress]{Division of Engineering and Mathematics, University of Washington Bothell, Bothell, WA 98011-8246, USA}
\address[uwb]{University of Washington Bothell, Bothell, WA 98011-8246, USA}
\fntext[funding2]{Supported by a 2021 University of Washington Bothell Scholarship, Research, and Creative Practice Seed Grant.}

\begin{abstract}
In further pursuit of a solution to the celebrated nonnegative inverse eigenvalue problem, Loewy and London [Linear and Multilinear Algebra 6 (1978/79), no.~1, 83--90] posed the problem of characterizing all polynomials that preserve all nonnegative matrices of a fixed order. If $\mathscr{P}_n$ denotes the set of all polynomials that preserve all $n$-by-$n$ nonnegative matrices, then it is clear that polynomials with nonnegative coefficients belong to $\mathscr{P}_n$. However, it is known that $\mathscr{P}_n$ contains polynomials with negative entries. In this work, novel results for $\mathscr{P}_n$ with respect to the coefficients of the polynomials belonging to $\mathscr{P}_n$. Along the way, a generalization for the even-part and odd-part are given and shown to be equivalent to another construction that appeared in the literature. Implications for further research are discussed.
\end{abstract}


\end{frontmatter}

% \linenumbers

%---------------------
\section{Introduction}

The longstanding \emph{nonnegative inverse eigenvalue problem} (``NIEP") is the problem of recognizing the spectra of nonnegative matrices. In particular, given a multi-set $\Lambda = \{ \lambda_1, \dots, \lambda_n \}$ of complex numbers, the NIEP asks for necessary and sufficient conditions such that $\Lambda$ is the spectrum of an $n$-by-$n$ entrywise-nonnegative matrix $A$, in which case the multi-set $\Lambda$ is called \emph{realizable} and the matrix $A$ is called a \emph{realizing matrix} for $\Lambda$. 

It is well-known that if $\Lambda$ is realizable with realizing matrix $A$, then  
\begin{equation*} 
s_k (\Lambda) := \sum_{i=1}^n \lambda_i^k = \trace{(A^k)} \geq 0,~\forall k \in \mathbb{N}, 	%\label{trnn}		
\end{equation*}
and  
\begin{equation}
s_k^m (\Lambda) \leq n^{m-1} s_{km} (\Lambda), \forall k, m \in \mathbb{N}.  		           \label{JLL}  
\end{equation}
Condition \eqref{JLL} is called the \emph{J-LL condition} and was established independently by Johnson \cite{j1981} and by Loewy and London \cite{ll1978-79}.

For a polynomial $p$, let $p(\Lambda) := \{ p(\lambda_1), \dots, p(\lambda_n) \}$. If $p$ is a polynomial that preserves the nonnegativity of all $n$-by-$n$ nonnegative matrices, then  
\begin{equation} 
s_k (p(\Lambda)) = \sum_{i=1}^n p(\lambda_i)^k = \trace{(p(A)^k)} \geq 0,~\forall k \in \mathbb{N}, 	\label{trnnp}		
\end{equation}
and
\begin{equation}
s_k^m (p(\Lambda)) \leq n^{m-1} s_{km} (p(\Lambda)), \forall k, m \in \mathbb{N}.  		           \label{JLLp}  
\end{equation}

As noted by Loewy and London \cite{ll1978-79}, determining whether the inequalities \eqref{trnnp} and \eqref{JLLp} are sufficient requires a characterization of   
\[\mathscr{P}_n := \left\{ p \in \mathbb{C}[x] \mid p(A) \geq 0, \forall A \in \mathsf{M}_n (\mathbb{R}), A \geq 0  \right\}. \]
In particular, and for practical purposes, necessary and sufficient conditions are sought in terms of the coefficients of the polynomials. However, few results are known with respect to the coefficients of the polynomials in $\mathscr{P}_n$.

The characterization of $\mathscr{P}_1$ is known as the P\'{o}lya--Szeg\"{o} theorem (see, e.g., Powers and Reznick \cite[Proposition 2]{pr2000}), which asserts that $p \in \mathscr{P}_1$ if and only if 
\[ p(x) = \left(f_1(x)^2 +f_2(x)^2\right) + x\left(g_1(x)^2 + g_2(x)^2\right), \]
where $f_1,f_2,g_1,g_2 \in \mathbb{R}[x]$.

Bharali and Holtz \cite{bh2008} gave partial results for the set 
\[\mathscr{F}_n := \left\{ f~\text{entire} \mid f(A) \geq 0, \forall A \in \mathsf{M}_n (\mathbb{R}), A \geq 0  \right\} \supset \mathscr{P}_n \]
and characterized entire functions that preserve certain structured nonnegative matrices, including upper-triangular matrices and circulant matrices.

In this work, novel results on the coefficients of polynomials in $\mathscr{P}_n$ are presented. Along the way, a generalization for the even-part and odd-part is given and shown to be equivalent to a construction given by Balaich and Ondrus \cite{bo2011} for polynomials. We conclude with implications for further inquiry. 

%--------------------------------
\section{Notation and Background}

The set of $m$-by-$n$ matrices with entries from a field $\mathbb{F}$ is denoted by $\mathsf{M}_{m\times n}(\mathbb{F})$. If $m = n$, then $\mathsf{M}_{m\times n}(\mathbb{F})$ is abbreviated to $\mathsf{M}_{n}(\mathbb{F})$. The set of all $n$-by-$1$ column vectors is identified with the set of all ordered $n$-tuples with entries in $\mathbb{F}$ and thus denoted by $\mathbb{F}^n$. The $n$-by-$n$ identity matrix is denoted by $I = I_n$.

If $n \in \mathbb{N}$, $n > 1$, and $\lambda \in \mathbb{C}$, then $J_n(\lambda)$ denotes the \emph{Jordan block with eigenvalue $\lambda$}, i.e., 
\[ 
J_n(\lambda) =
\begin{bmatrix}
\lambda & 1         &                                   \\
        & \lambda   & 1                                 \\
        &           & \ddots    & \ddots                \\
        &           &           & \lambda   & 1         \\
        &           &           &           & \lambda
\end{bmatrix} \in \mathsf{M}_n(\mathbb{C}). \]

If $A \in \mathsf{M}_n(\mathbb{F})$, then $a_{ij}$ denotes the $(i,j)$-entry of $A$. If $\mathbb{F} = \mathbb{R}$ and $a_{ij} \ge 0$, $1 \le i,j \le n$, then $A$ is called \emph{nonnegative} and this is denoted by $A \geq 0$. 

Unless otherwise stated,   
\[ p(x) = \sum_{k=0}^m a_k x^k \in \mathbb{C}[x],\]
where $a_m \ne 0$. If $n$ is a positive integer less than or equal to $m$, then the coefficients $a_0, a_1,\dots, a_{n-1}$ are called the \emph{first $n$ terms of $p$} and the coefficients $a_{m-n + 1}, \dots, a_{m-1}, a_m$ are called the \emph{last $n$ terms of $p$}.

If $A \in \mathsf{M}_n(\mathbb{F})$, then $A$ is called a \emph{circulant} or \emph{circulant matrix} if there is a vector $v \in \mathbb{F}^n$, called the \emph{reference vector of $A$}, such that $a_{ij} = v_{(j-i)\bmod{n}+1}$. In such a case, we write $A = \circulant(v_1, \ldots, v_n)$. 
%It is well-known that $A$ is a circulant if and only if there is a diagonal matrix $D$ such that $A = FDF^*$ \cite[Theorems 3.2.2 and 3.2.3]{d1979}. 
The $n$-by-$n$ circulant matrix with reference vector $e_2$ is denoted by $C = C_n$. Finally, note that if $q(x) = \sum_{k=0}^{n-1} v_{k+1} x^k$, then $A = q(C) = \sum_{k=0}^{n-1} v_{k+1} C^k$ \cite[p.~68]{d1979}.

%-------------------------------------------------------------------
\section{The Even and Odd Part of a Polynomial and a Generalization}

In this section, we generalize the notion of the even and odd part of a polynomial and show that it is equivalent to a construction given by Balaich and Ondrus \cite{bo2011}.

Recall that if $f: \mathbb{C} \longrightarrow \mathbb{C}$, then 
\[ f_e(x) := \frac{f(x) + f(-x)}{2} \]
is called the \emph{even-part of $f$} and  
\[ f_o(x) := \frac{f(x) - f(-x)}{2} \]
is called the \emph{odd-part of $f$}.  

As is well-known, or otherwise easy to show,   
\[ p_e(x) = \sum_{k\equiv 0\bmod{n}} a_k x^k \] 
and 
\[ p_o(x) = \sum_{k\equiv 1\bmod{n}} a_k x^k. \]
The following construction generalizes the aforementioned functions.

%------------
\begin{mydef}
\label{indexsets}
Let $n \in \mathbb{N}$ and $r \in \{0,1,\ldots,n-1\}$. If   
\[ \mathcal{I}_{(m,n,r)} := \{0 \le k \le m  \mid k\bmod{n} = r \}, \]
then the polynomial 
\[ p_{(r,n)}(x) := \sum_{k \in \mathcal{I}_{(m,n,r)}} a_k x^k, \]
is called the \emph{$r \bmod{n}$-part of $p$}.
\end{mydef}

%----------
\begin{obs}
    \label{obs:rpartsum}
If $p \in \mathbb{C}[x]$, then 
\(
    p(x) = \sum_{r = 0}^{n-1} p_{(r,n)}(x). %\label{polynomialdecomposition}    
\)
\end{obs}

\begin{proof}
Follows from the fact that if $r \ne \hat{r}$, then $\mathcal{I}_{(m,n,r)} \cap \mathcal{I}_{(m,n,\hat{r})} = \emptyset$ and the fact that 
\[ \bigcup_{r=0}^{n-1} \mathcal{I}_{(m,n,r)} = \{0,1,\ldots,m\}. \qedhere \]
\end{proof}

Balaich and Ondrus offered the following as a generalization to the even-part and odd-part of a function. 

%------------
\begin{mydef}
    [{\cite[Definition 3]{bo2011}}]
If $f: \mathbb{C} \longrightarrow\mathbb{C}$, $r \in \{0,1,\dots,n-1\}$, and $\omega := \exp(2 \pi i/ n)$, then
\begin{equation}
\label{m2011function}
    f_{(r,\omega)}(z) := \frac{1}{n} \sum_{k=0}^{n-1} \omega^{-kr} f(\omega^k z)   
\end{equation}
is called the \emph{$r \bmod{n}$-part of $f$}.
\end{mydef}   

In the case of polynomials, the following result shows that the constructions yield the same function.

%----------
\begin{thm}
    \label{thm:unitypoly}
If $p \in \mathbb{C}[x]$, then $p_{(r,n)}(x) = p_{(r,\omega)}(x)$.
\end{thm}

\begin{proof} 
Because 
\[ \sum_{k=0}^{n-1} \omega^{jk} = \begin{cases}
n, & j \equiv 0\bmod{n} \\
0, & j \not \equiv 0\bmod{n}
\end{cases}, \]
it follows that
\begin{align*}
    p_{(r,\omega)}(x) 
    &= \frac{1}{n} \sum_{k=0}^{n-1} \omega^{-kr} p(\omega^k x) \\
    &= \frac{1}{n} \sum_{k=0}^{n-1} \omega^{-kr} \sum_{j=0}^m a_j \omega^{jk} x^j \\
    &= \frac{1}{n} \sum_{k=0}^{n-1} \sum_{j=0}^m a_j x^j \omega^{k(j-r)} \\
    &= \frac{1}{n} \sum_{j=0}^m a_j x^j \sum_{k=0}^{n-1} \omega^{k(j-r)} \\
    &= \frac{1}{n} \left( \sum_{j \in \mathcal{I}_{(m,n,r)}} a_j x^j \sum_{k=0}^{n-1} \omega^{k(j-r)} + \sum_{j \notin \mathcal{I}_{(m,n,r)}} a_j x^j \sum_{k=0}^{n-1} \omega^{k(j-r)} \right) \\
    &= \frac{1}{n} \left( \sum_{j \in \mathcal{I}_{(m,n,r)}} a_j x^j (n) + \sum_{j \notin \mathcal{I}_{(m,n,r)}} a_j x^j 0 \right) \\
    &= \sum_{j \in \mathcal{I}_{(m,n,r)}} a_j x^j = p_{(r,n)}(x). \qedhere
\end{align*}
\end{proof}

%-------------------------------------------------------
\section{Coefficients of polynomials in $\mathscr{P}_n$}

Clark and Paparella recently showed that the coefficients of any polynomial belonging to $\mathscr{P}_n$ must be real \cite[Corollary 3.4]{clark2021polynomials}. As such, hereinafter it is assumed that the coefficients of $p$ are real.

Bharali and Holtz \cite[Theorem 6]{bh2008} established the following result for entire functions using block upper-triangular matrices. We offer a simpler proof using Jordan blocks that also applies to {entire} functions.

%----------
\begin{thm}
    \label{thm:containDerivate}
If $p \in \mathscr{P}_n$, then $p, p^{(1)}, p^{(2)},\ldots, p^{(n-1)} \in \mathscr{P}_1$.
\end{thm}

\begin{proof}
The result follows by utilizing the well known fact \cite[p.~386]{hj1994} that 
\[
p(J_n(t))= \kbordermatrix{
        & 1      & \cdots & k                           & \cdots & n                            \\
 1      & p(t)   & \cdots & \frac{p^{(k-1)}(t)}{(k-1)!} & \cdots & \frac{p^{(n-1)}(t)}{(n-1)!}  \\
 \vdots &        & \ddots &                             & \ddots &                              \\
 k      &        &        & p(t)                        &        & \frac{p^{(k-1)}(t)}{(k-1)!}  \\
 \vdots &        &        &                             & \ddots &                              \\
 n      &        &        &                             &        & p(t)                         \\
}. \qedhere \]
\end{proof}

%----------
\begin{cor}
\label{cor:firstnterms}
If $p \in \mathscr{P}_n$ and $m \ge n-1$, then the first $n$ terms of $p$ are nonnegative; otherwise, if $m < n-1$, then all of its coefficients are nonnegative. 
\end{cor}

\begin{proof}
Follows from the fact that \( a_k = {p^{(k)}(0)}/{k!} \ge 0\).
\end{proof}

%----------
\begin{obs}
\label{lem:operations}
If $p,q \in \mathscr{P}_n$ and $\alpha \ge 0$, then $\alpha p$, $p+q$, $pq$, and $p \circ q \in \mathscr{P}_n$.
\end{obs}

\begin{proof}
Follows from the fact that $(p+q)(A) = p(A) + q(A)$ \cite[Theorem 1.15]{h2008}, $(pq)(A) = p(A) q(A) $ \cite[Theorem 1.15]{h2008}, and $ (p\circ q)(A) =  p(q(A))$ \cite[Theorem 1.17]{h2008}.  
\end{proof}

Since $C^n = I$ and $C$ is invertible, it follows by the division algorithm that $C^k = C^r$, where $r = k\bmod{n}$. 

%----------
\begin{prop}
    \label{prop:scaledcirculant}
If $p \in \mathbb{C}[x]$ and $t \in \mathbb{C}$, then 
\[ p(tC) = \sum_{r=0}^{n-1} p_{(r,n)}(t) C^r = \circulant\left( p_{(0,n)}(t), p_{(1,n)}(t),\dots, p_{(n-1,n)} (t)\right). \]
\end{prop}

\begin{proof}
By Observation \ref{obs:rpartsum} and the fact that $(p+q)(A) = p(A) + q(A)$, 
\begin{align*}
    p(tC) 
    = \sum_{r=0}^{n-1} p_{(r,n)} (tC)                                                                              
    &= \sum_{r=0}^{n-1} \sum_{k \in \mathcal{I}_r} a_k t^k C^k                              \\
    &= \sum_{r=0}^{n-1} \sum_{k \in \mathcal{I}_r} a_k t^k  C^r                             \\
    &= \sum_{r=0}^{n-1} p_{(r,n)}(t) C^r                                                    \\
    &= \text{circ}\left( p_{(0,n)}(t), p_{(1,n)}(t),\dots, p_{(n-1,n)} (t)\right).   \qedhere
\end{align*}
\end{proof}

%----------
\begin{cor}
    \label{cor:polyparts}
If $p \in \mathscr{P}_n$, then $p_{(r,n)} \in \mathscr{P}_1,~\forall r \in \{0,1,\dots,n-1\}$.
\end{cor}

\begin{proof}
If $p \in \mathscr{P}_n$ and $t \ge 0$, then 
\[p(tC) = \circulant\left(p_{(0,n)}(t), \ldots, p_{(n-1,n)}(t)\right) \ge 0, \]
and the latter holds if and only if $p_{(r,n)} \in \mathscr{P}_1$, $\forall r \in \{0,1,\dots,n-1\}$.
\end{proof}

%----------
\begin{cor}
    \label{cor:modcoeff}
If $p \in \mathscr{P}_n$, then   
\[  \sum_{k \in \mathcal{I}_{(m,n,r)}} a_k \geq 0,~\forall r \in \{0,1,\ldots,n-1\}.\]
\end{cor}

\begin{proof}
If $p \in \mathscr{P}_n$, then 
\[ \sum_{k \in \mathcal{I}_{(m,n,r)}} a_k = p_{(r,n)}(1) \geq 0,~\forall r \in \{0,1,\ldots,n-1\}. \qedhere \]
\end{proof}

\begin{rem}
Since $a_r = p_{(r,n)} (0) \ge 0$, Corollary \ref{cor:polyparts} yields another proof of Corollary \ref{cor:firstnterms}, which can also be established by examining $p(J_n(0))$ \cite[Proposition 2]{bh2008}.
\end{rem}

%----------
\begin{lem}
    \label{lem:contain}
    If $n \in \mathbb{N}$, then $\mathscr{P}_{n+1} \subseteq \mathscr{P}_n$.
\end{lem}

\begin{proof}
For completeness, we repeat the argument given by Bharali and Holtz \cite[Lemma 1]{bh2008} for the containment $\mathscr{F}_{n+1} \subseteq \mathscr{F}_n$: let $A$ be a nonnegative matrix of order $n$ and let $p \in \mathscr{P}_{n+1}$. If $B : = \diag(A, 0) \in \mathsf{M}_{n+1} (\mathbb{R})$, then $p(B) = \diag(p(A), 0)$ \cite[Theorem 1.13(g)]{h2008}. Since $p(B) \ge 0$, it follows that $p(A) \ge 0$, i.e., $p \in \mathscr{P}_n$. 
\end{proof}

%----------
\begin{cor}
If $p \in \mathscr{P}_n$, then $p_{(r,\hat{n})} \in \mathscr{P}_1$ for every $\hat{n} \in \{1,\ldots, n\}$ and for every $r \in \{0,1,\dots,\hat{n}-1\}$.
\end{cor}

\begin{proof}
Immediate from Corollary \ref{cor:polyparts} and Lemma \ref{lem:contain}.
\end{proof}

%----------
\begin{cor}
If $p \in \mathscr{P}_n$, then   
\[  \sum_{k \in \mathcal{I}_{(m,\hat{n},r)}} a_k \geq 0, \]
for every $\hat{n} \in \{1,\ldots, n\}$ and for every $r \in \{0,1,\dots,\hat{n}-1\}$.
\end{cor}

\begin{proof}
Immediate from Corollary \ref{cor:modcoeff} and Lemma \ref{lem:contain}.
\end{proof}

%----------
\begin{thm}
\label{thm:lastnterms}
If $p \in \mathscr{P}_n$ and $\deg p > n$, then the last $n$ terms are nonnegative.
\end{thm}

\begin{proof}
By Proposition \ref{prop:scaledcirculant}, if $t \in \mathbb{R}$, then 
\[ p(tC) = \circulant(p_{(0,n)}(t), \ldots, p_{(n-1,n)}(t)) \ge 0. \]
For contradiction, suppose that $a_{m-n+k} < 0$, where $1 \le k \le n$. The $n$ terms 
\[ {m-n+1}, \ldots, {m-1}, m \] 
form a \emph{complete residue system modulo $n$} since they are consecutive; as such, if $r = (m-n+k)\bmod{n}$, then $a_{m-n+k}$ is the leading coefficient of $p_{(r,n)}$. Since $a_{m-n+k} < 0$, we may select $t$ large enough such that $p_{(r,n)}(t) < 0$, a contradiction. Thus, the last $n$ terms of $p$ must be nonnegative.
\end{proof}

%----------
\begin{thm}
    \label{thm:degcond}
If $p \in \mathbb{C}[x]$ and $\deg p < 2n$, then $p \in \mathscr{P}_n$ if and only if all of the coefficients are nonnegative.
\end{thm}

\begin{proof}
Sufficiency is clear and necessity follows from Corollary \ref{cor:firstnterms} and Theorem \ref{thm:lastnterms}.
\end{proof}

%---------------------------
\section{Concluding Remarks}

In addition to determining whether conditions presented previously are sufficient (or to disprove otherwise), we offer additional results and several other lines of inquiry. 

% %----------
% \begin{lem}
%     \label{thm:contain}
% If $n \in \mathbb{N}$, then $\mathscr{P}_{n+1} \subseteq \mathscr{P}_{n}$.
% \end{lem}

% \begin{proof}
% Follows by a similar argument given by Bharali and Holtz [Lemma 1]\cite{bh2008} for entire functions.
% \end{proof}

%----------
\begin{thm}
$\mathscr{P}_2 \subset \mathscr{P}_1$. 
\end{thm}

\begin{proof}
If $p(x) = x^2 - 4x + 4 = (x-2)^2$, then $p(x) \in \mathscr{P}_1$ since $p(x) \geq 0, \forall x \in \mathbb{R}$. However, $p(x) \notin \mathscr{P}_2$ in view of Theorem \ref{thm:degcond}.
\end{proof}

Clark and Paparella \cite[Theorem 5.6]{clark2021polynomials} recently showed that $\mathscr{P}_3 \subset \mathscr{P}_2$. Thus, we offer the following. 

%-----------
\begin{conj}
If $n \ge 3$, then $\mathscr{P}_{n+1} \subset \mathscr{P}_{n}$.
\end{conj}

If $V$ is a vector space over $\mathbb{R}$ and $U$ is a nonempty subset of $V$, then $U$ is called a \emph{convex cone} if $\alpha u + \beta v \in U$ for every $u$, $v \in U$ and $\alpha$, $\beta \geq 0$. A point $x$ of a convex cone $U$ is called an \emph{extreme direction (or ray)} of $U$ if, whenever $x = \alpha u + \beta v$, with $\alpha,\beta > 0$ and $u, v \in U$, then $x = u = v$. By Observation \ref{lem:operations}, $\mathscr{P}_n$ is a convex cone of $\mathbb{C}[x]$ and $\{ p_k(x) = x^k \in \mathbb{C}[x] \mid k \ge 0 \}$ consists of extreme directions. As such, we pose the following line of inquiry. 

\begin{prob}
Identify the extreme directions of $\mathscr{P}_n$.
\end{prob}

%--------------------------
\section*{Acknowledgements}

We thank the anonymous referee for their careful review and suggestions that improved this work and the the University of Washington Bothell for generous funding that supported this work. 

%------------------------
\bibliographystyle{abbrv}
\bibliography{nonpoly}

\end{document}